\documentclass[aps,prl,preprint,groupedaddress]{revtex4-1}

\usepackage{amsmath}
\usepackage{amsfonts}
\usepackage{amssymb}
\usepackage{amsthm}
\usepackage{textcomp}
\usepackage[utf8]{inputenc}
\usepackage{enumitem}

\newtheorem{theorem}{Theorem}
\newtheorem{corollary}{Corollary}[theorem]

\begin{document}


\title{On a problem of the theory of elasticity}


\author{Maksut M. Abenov}
\affiliation{Institute of Mechanics and Mechanical Engineering, National  Academy of Sciences, Republic of Kazakhstan, Almaty}
\email{abenov.m.m@gmail.com}
\author{Nourlan B. Shaltykov}%
\affiliation{Institute of Fundamental Researches, Republic of Kazakhstan, Almaty}
\email{nshaltykov@gmail.com}


\date{\today}

\begin{abstract}
The exact solution of the Cauchy problem of the linear theory of elasticity is given in the paper, when the initial data belong to a specific class of functions.
\end{abstract}


\maketitle

\section{Statement of the problem}
Let $ \vec {U} = (u_1 (t, x_1, x_2, x_3), u_2, u_3) $ be the unknown displacement vector, $ \rho, \lambda, \mu $ be the known continuous Lamé constants. Investigated continuous medium is in an elastically deformed state under the action of a known mass force $\vec{F}=(F_1(t,x_1,x_2,x_3),F_2,F_3)$.

We consider the system of equations of the theory of elasticity \cite{Gr-Ze}:
\begin{equation}\label{NS}
\mu\Delta\vec{U}+ (\lambda+\mu) \nabla (div\vec{U}) + \rho\vec{F} =\rho \frac{\partial^2\vec{U}}{\partial t^2} \\
\end{equation}
in the $ G =  R^3\times [0, \infty) $ region with the following initial conditions:
\begin{equation}\label{init1}
\vec{U}(0,x_1,x_2,x_3)=\vec{\varphi}=(\varphi_1(x_1,x_2,x_3),\varphi_2,\varphi_3)
\end{equation}
\begin{equation}\label{init2}
\frac{\partial\vec{U}}{\partial t}(0,x_1,x_2,x_3)=\vec{\psi}=(\psi_1(x_1,x_2,x_3),\psi_2,\psi_3)
\end{equation}
It is necessary to determine sufficient conditions for the initial data of problem (1) - (3) for which there exist bounded smooth solutions $\vec{U}\in C^\infty (G)$ of the Cauchy problem with bounded smooth components of the corresponding stress tensor $\tau_{ij}\in C^\infty (G), i,j=\overline{1,3}$.
\section{Basic theorem}
\begin{theorem}[Basic theorem]
	Let the initial data of the problem satisfy the following conditions:
	\begin{enumerate}[label=\Alph*.]
		\item $\vec{F}\in C^\infty (G),\vec{\varphi},\vec{\psi}\in C^\infty (R^3)$ and the components of these vectors vanish at infinity, along with its derivatives of any order.
		\item The following relations hold:
		\begin{equation}\label{b1}
		rot\vec{F}=0;\quad  rot\vec{\varphi}=0;\quad  rot\vec{\psi}=0;
		\end{equation}
		\begin{equation}\label{b2}
		\frac{\partial\varphi_1}{\partial x_1}=\frac{\partial\varphi_2}{\partial x_2}=\frac{\partial\varphi_3}{\partial x_3}
		\end{equation}
		\begin{equation}\label{b3}
		\frac{\partial\psi_1}{\partial x_1}=\frac{\partial\psi_2}{\partial x_2}=\frac{\partial\psi_3}{\partial x_3}
		\end{equation}	
	\end{enumerate}
\end{theorem}
Then:
\begin{corollary}
	The only bounded solution of the problem (1) - (3) $ \vec {U} \in C^\infty (G) $ is given by the following formulas:
	\begin{multline}\label{da1}
	u_1=\frac{\varphi_1(x_1+at,x_2,x_3)+\varphi_1(x_1-at,x_2,x_3)}{2}+\\
	\frac{1}{2a}\int\limits_{x_1-at}^{x_1+at}\psi_1(\alpha,x_2,x_3)d\alpha+\frac{1}{2a}\int\limits_{0}^{t}d\tau\int\limits_{x_1-a(t-\tau)}^{x_1+a(t-\tau)}F_1(\tau,\alpha,x_2,x_3)d\alpha
	\end{multline}
	\begin{multline}\label{da2}
	u_2=\frac{\varphi_2(x_1,x_2+at,x_3)+\varphi_1(x_1,x_2-at,x_3)}{2}+\\
	\frac{1}{2a}\int\limits_{x_2-at}^{x_2+at}\psi_2(x_1,\beta,x_3)d\beta+\frac{1}{2a}\int\limits_{0}^{t}d\tau\int\limits_{x_2-a(t-\tau)}^{x_2+a(t-\tau)}F_2(\tau,x_1,\beta,x_3)d\beta
	\end{multline}
	\begin{multline}\label{da3}
	u_3=\frac{\varphi_3(x_1,x_2,x_3+at)+\varphi_3(x_1,x_2,x_3-at)}{2}+\\
	\frac{1}{2a}\int\limits_{x_3-at}^{x_3+at}\psi_3(x_1,x_2,\gamma)d\gamma+\frac{1}{2a}\int\limits_{0}^{t}d\tau\int\limits_{x_3-a(t-\tau)}^{x_3+a(t-\tau)}F_3(\tau,x_1,x_2,\gamma)d\gamma
	\end{multline}
\end{corollary}
\begin{corollary}
	The components of the stress tensor $ \tau_{ij} \in C^\infty (G), i, j = \overline {1,3} $ corresponding to this solution, will also be bounded functions, which follows from the formula:
	\begin{equation}\label{tau}
	\tau_{ij}=\lambda \delta_{ij}div\vec{U}+\mu (\frac{\partial u_i}{\partial x_j}+\frac{\partial u_j}{\partial x_i});\quad i,j=\overline{1,3} 
	\end{equation}
	where $\delta_{ij}$ is Kronecker symbol, $a=\sqrt{\frac{3(\lambda+2\mu)}{\rho}}$.
\end{corollary}
\begin{proof}
	It is easy to understand from (\ref{b1})-(\ref{b3}) that the components of the required displacement vector must satisfy the following conditions:
	\begin{equation}\label{11}
	\frac{\partial u_1}{\partial x_2}=\frac{\partial u_2}{\partial x_1};\quad\frac{\partial u_1}{\partial x_3}=\frac{\partial u_3}{\partial x_1};\quad\frac{\partial u_2}{\partial x_3}=\frac{\partial u_3}{\partial x_2}
	\end{equation}
	\begin{equation}\label{12}
	\frac{\partial u_1}{\partial x_1}=\frac{\partial u_2}{\partial x_2}=\frac{\partial u_3}{\partial x_3}
	\end{equation}
	Then from (\ref{12}) one can obtain the following relation:
	\begin{equation}\label{13}
	div \vec{U}=\frac{\partial u_1}{\partial x_1}+\frac{\partial u_2}{\partial x_2}+\frac{\partial u_3}{\partial x_3}=3\frac{\partial u_1}{\partial x_1}=3\frac{\partial u_2}{\partial x_2}=3\frac{\partial u_3}{\partial x_3}
	\end{equation}
	Further, considering (\ref{11}) and (\ref{12}) together, one can obtain:
	\begin{equation}\label{14}
	\frac{\partial^2 u_k}{\partial x_1^2}=\frac{\partial^2 u_k}{\partial x_2^2}=\frac{\partial^2 u_k}{\partial x_3^2},
	\end{equation}
	This implies:
	\begin{equation}\label{15}
	\Delta u_k=3\frac{\partial^2 u_k}{\partial x_1^2}=3\frac{\partial^2 u_k}{\partial x_2^2}=3\frac{\partial^2 u_k}{\partial x_3^2},\quad k=\overline{1,3}
	\end{equation}
	The obtained relations (\ref{13}) and (\ref{15}) allow us to rewrite equation (\ref{NS}) for the unknown component $ u_1 $ in the following form:
	$$3\mu\frac{\partial^2 u_1}{\partial x_1^2}+ (\lambda+\mu) \frac{\partial}{\partial x_1} (3\frac{\partial u_1}{\partial x_1}) + \rho F_1 =\rho \frac{\partial^2 u_1}{\partial t^2}$$
	Or, after conversion:
	\begin{equation}\label{w1}
	\frac{\partial^2 u_1}{\partial t^2}=\frac{3(\lambda+2\mu)}{\rho}\frac{\partial^2 u_1}{\partial x_1^2}+  F_1  
	\end{equation}
	\begin{equation}\label{w2}
	\frac{\partial^2 u_2}{\partial t^2}=\frac{3(\lambda+2\mu)}{\rho}\frac{\partial^2 u_2}{\partial x_2^2}+  F_2  
	\end{equation}
	\begin{equation}\label{w3}
	\frac{\partial^2 u_3}{\partial t^2}=\frac{3(\lambda+2\mu)}{\rho}\frac{\partial^2 u_3}{\partial x_3^2}+  F_3  
	\end{equation}
	Each of equations (\ref{w1}) - (\ref{w3}) is considered as a one-dimensional wave equation, with two of the three spatial variables being considered as constant parameters. Then, the (\ref{init1})-(\ref{init2}) conditions are initial data for them. As is known, the solution of these one-dimensional Cauchy problem is given by the well-known d'Alembert formulas (\ref{da1})-(\ref{da3}). From these formulas and the properties of the initial data, it is easy to see that $ u_k(t, x_1, x_2, x_3) \in C^\infty (G),\quad k = \overline {1,3} $ are bounded functions that vanish at infinity, together with its all derivatives. The boundedness of the components of the stress tensor defined by (\ref{tau}) is trivial.
\end{proof}
\section{Conclusion}
It can be proven that the class of vector-valued functions satisfying conditions (\ref{b1}) - (\ref{b3}) is an infinite-dimensional linear space. For initial data from this class, the mixed initial-boundary problems in bounded domains is correct. Exact solutions of such problems can also be obtained by modifying the above described method.

\end{document}